\documentclass[11pt]{article}

\usepackage{amsmath,amssymb,amsfonts,amsthm}
\usepackage[margin=2.4cm]{geometry}
\usepackage[shortlabels]{enumitem}
\usepackage{microtype}
\usepackage{hyperref}
\usepackage[capitalise]{cleveref}
\usepackage{aliascnt}

\usepackage{xcolor}

\newtheorem{theorem}{Theorem}[section]
\newaliascnt{lemma}{theorem}
\newtheorem{lemma}[lemma]{Lemma}
\aliascntresetthe{lemma}
\newtheorem{prop}[theorem]{Proposition}

\newtheorem{Observation}[theorem]{Observation}
\newtheorem{fact}[theorem]{Fact}

\newaliascnt{claim}{theorem}
\newtheorem{claim}[claim]{Claim}
\aliascntresetthe{claim}

\crefname{claim}{Claim}{Claims}
\Crefname{claim}{Claim}{Claims}

\theoremstyle{definition}
\newtheorem{defn}[theorem]{Definition}
\newtheorem{rmk}[theorem]{Remark}
\crefname{lemma}{Lemma}{Lemmas}
\Crefname{lemma}{Lemma}{Lemmas}

\newcommand{\F}{\mathcal F}

\newcommand{\sq}{\sqcup}

\title{The maximum $t$-diversity of $t$-intersecting families of permutations}
\author{
Yuhang Cao\thanks{School of Mathematical Sciences, Capital Normal University, Beijing, China. Email: 2260501024@cnu.edu.cn, gnge@zju.edu.cn, 2250501013@cnu.edu.cn. Gennian Ge is supported by the National Key Research and Development Program of China under Grant 2025YFC3409900, the National Natural Science Foundation of China under Grant 12231014, and Beijing Scholars Program.}
\and
Gennian Ge\footnotemark[1]
\and
Jian Wang\thanks{School of Mathematics, Sichuan University,
Chengdu, China. Email: wangjianmath01@scu.edu.cn. Jian Wang is supported by National Natural Science
Foundation of China Grant no. 12471316.}
\and
Jialuo Wang\thanks{School of Mathematical Sciences, University of Science and Technology of China,
Hefei, China. Email: 673391676@qq.com}
\and
Xiaochen Zhao\footnotemark[1]
}
\date{September, 2026}

\begin{document}
\maketitle

\begin{abstract}
The study of $t$-intersecting families in symmetric groups received lots of attention in the last two decades. In this paper, we study two different kinds of stability results for $t$-intersecting families in symmetric group. Let $\Sigma_n$ denote the symmetric group on $\{1,2,\ldots,n\}$. 
The $t$-diversity $\gamma_t(\F)$ of  $\F\subseteq\Sigma_n$ is defined as the minimum number of members of $\F$, whose deletion results in a family with transversal number $t$. The star $t$-diversity $\gamma_t^{\star}(\mathcal{F})$ is defined as the minimum number of members of $\F$, whose deletion results in a $t$-star.  For $n$ relatively large with respect to $t$, we determine the best possible bounds for both $\gamma_t(\F)$ and $\gamma_t^{\star}(\mathcal{F})$ over all $t$-intersecting families $\F\subseteq\Sigma_n$. The equality holding conditions are characterized. For $\gamma_t(\F)$, the extremal case is generated by all $2t$-subsets of a $3t$-partial permutation or, when $t=2$, the family of complements of the lines of the Fano plane. For $\gamma_t^{\star}(\mathcal{F})$, the extremal family is generated by all $(t+1)$-subsets of a $(t+2)$-partial permutation.
\end{abstract}

\section{Introduction}

The study of intersecting families is one of the central themes of extremal
set theory. Let $[n]=\{1,2,\ldots,n\}$ and $\binom{[n]}{k}$ denote the family of all $k$-element subsets of $[n]$. A family $\mathcal A\subseteq \binom{[n]}{k}$ is called intersecting if any two of its members have non-empty intersection, and more generally it is called $t$-intersecting if any two of its members intersect in at least $t$ points.  The classical theorem of Erd\H{o}s, Ko and Rado determines the maximum size of an intersecting subfamily of $\binom{[n]}{k}$ when $n$ is sufficiently large with respect to $k$ \cite{EKR}. The exact bound in the $t$-intersecting version was proved by Frankl\cite{Frankl77} for $t\ge 15$, and subsequently by Wilson\cite{Wilson84} for all $t$, and the complete intersection problem for all parameters was settled by Ahlswede and Khachatrian \cite{AhlswedeKhachatrian97}.  Alongside these exact extremal theorems, non-triviality, stability, degree conditions and shifting methods have played a fundamental role; see, for example, the Hilton--Milner theorem \cite{HiltonMilner67} and work of Frankl and his collaborators \cite{Frankl78,Frankl77,FranklShifting87,Frankl87,FranklFuredi86}.

It is natural to ask for Erd\H{o}s--Ko--Rado type theorems in other highly symmetric discrete spaces.  One of the most important such settings is the symmetric group.  Let $\Sigma_n$ denote the family of all permutations of $[n]$.  We identify a permutation $\sigma\in\Sigma_n$ with the perfect matching
\[
        \{(i,\sigma(i)): i\in[n]\}\subset [n]\times[n].
\]
Thus two permutations are $t$-intersecting if, as subsets of $[n]\times[n]$, they have at least $t$ common points; equivalently, they agree on at least $t$ positions.  Deza and Frankl initiated the permutation analogue of the Erd\H{o}s--Ko--Rado theorem and proved that an intersecting family in $\Sigma_n$ has size at most $(n-1)!$ \cite{DezaFrankl77}.  Cameron and Ku proved the corresponding strict EKR property, showing that the maximum intersecting families are precisely the cosets of point stabilizers \cite{CameronKu03}; see also \cite{LaroseMalvenuto04,GodsilMeagher09}.  Ellis, Friedgut and Pilpel proved the Deza--Frankl conjecture for $t$-intersecting families of permutations when $n$ is sufficiently large depending on $t$ \cite{EllisFriedgutPilpel11}.  More recently, major progress on forbidden intersection problems for permutations has been obtained through stability, junta, hypercontractive and spread approximation methods \cite{EllisKellerLifshitz24,KellerLifshitzMinzerSheinfeld24,KUPAVSKII2024109653,Kupavskii24b}.

The question of $t$-diversity has a long history. For the case of $ t=1$, the diversity of an intersecting family  $\mathcal F\subset \binom{[n]}{k}$ is
\[
        \gamma(\mathcal F)=\min_{x\in[n]} |\{F\in\mathcal F:x\notin F\}|=|\mathcal F|-\Delta(\mathcal F).
\]
where
\[
        \Delta(\mathcal F)=\max_{x\in[n]} |\{F\in\mathcal F:x\in F\}|,
\]
Equivalently, $\gamma(\mathcal F)$ measures the distance from $\mathcal F$ to a star, which is the trivial Erd\H{o}s--Ko--Rado theorem extremal examples.  This parameter was originally called the unbalance of set systems by Lemons and Palmer \cite{LemonsPalmer08}.  Kupavskii proved a conjecture of Frankl on the maximum diversity of uniform intersecting families for large $n$ \cite{Kupavskii18}.  Frankl subsequently improved the range of validity using maximum-degree methods \cite{Frankl20}, and Frankl and Kupavskii developed a more systematic treatment of diversity \cite{FranklKupavskii21}.  Further refinements and variants were obtained by Frankl and the third author \cite{FranklWang24,FranklWang24MaxDegree,FranklWang25Cdiversity}.  Related notions, including $t$-diversity and generalized diversity, have also been studied \cite{KuWong20,MagnanPalmerWood24}.

Xiao and the third author recently studied the diversity for intersecting families in the symmetric group \cite{WangXiao25}. The diversity of intersecting family $\mathcal F\subset\Sigma_n$ is
\[
        \gamma_1(\mathcal F)=
        \min_{(i,j)\in[n]\times[n]}
        |\{\sigma\in\mathcal F:\sigma(i)\ne j\}|.
\]
They proved the following sharp result.

\begin{theorem}[{\cite{WangXiao25}}]\label{thm:wang-xiao}
Let $\mathcal F\subset\Sigma_n$ be an intersecting family.  If $n\ge 500$, then
\[
        \gamma_1(\mathcal F)\le (n-3)(n-3)!.
\]
Moreover, the bound is tight.
\end{theorem}

Our goal is to extend this theorem from intersecting families to $t$-intersecting families of permutations.  Since $\Sigma_n$ is a sparse and highly dependent subfamily of $\binom{[n]\times[n]}{n}$, diversity results for ordinary uniform set systems do not apply directly.  The matching structure of permutations must be used essentially.  Our proof combines the spread approximation method of Kupavskii and Zakharov with a finite configuration analysis of small partial permutations.

We now introduce the notation used throughout the paper.  A set $S\subset [n]\times[n]$ is called a partial permutation if it is contained in some permutation $\sigma\in\Sigma_n$.  For $\mathcal F\subset\Sigma_n$ and a partial permutation $T$, define
\[
        \mathcal F(\overline T)=\{F\in\mathcal F: F\cap T=\emptyset\},
        \qquad
        \mathcal F[T]=\{F\in\mathcal F: T\subset F\},
        \qquad
        \mathcal F(T)=\{F \backslash T: T\subset F\in\mathcal F\}.
\]
If $S$ and $T$ are disjoint sets of pairs, we also write
\[
        \mathcal F[S,\overline T]
        =\{F\in\mathcal F:S\subset F,\ F\cap T=\emptyset\}.
\]
For a family $\mathcal B$ of partial permutations, define
\[
        \mathcal F[\mathcal B]=\bigcup_{B\in\mathcal B}\mathcal F[B].
\]

\begin{defn}\label{def:t-diversity}
Let $\mathcal F\subset\Sigma_n$.  The $t$-diversity of $\mathcal F$ is defined as
\[
        \gamma_t(\mathcal F)=
        \min_{\substack{T\text{ is a partial permutation}\\ |T|=t}}
        |\mathcal F(\overline T)|.
\]
\end{defn}

The $\gamma_t(\mathcal F)$ measures the distance from such a family to a family with a transversal of size $t$, where a transversal is a set intersecting with every elements in $\mathcal{F}$.  Our first main theorem gives the sharp upper bound for this parameter.

\begin{theorem}\label{thm:main-bound}
Let $t\ge1$, and let $\mathcal F\subset\Sigma_n$ be a $t$-intersecting family.  If $n\ge 130(t+1)(8t)^t$, then
\[
        \gamma_t(\mathcal F)\le \sum_{i=0}^{t}(-1)^i\binom{t}{i}(n-2t-i)!.
\]
\end{theorem}

For $t=1$, the right-hand side is $(n-3)(n-3)!$, so \Cref{thm:main-bound} covers \Cref{thm:wang-xiao} with a weaker numerical threshold.

Let $\mathcal{T}^*(X,t):=\{F\in\Sigma_n: |F\cap X|=2t\}$ and  $\mathcal{T}(X,t):=\{F\in\Sigma_n: |F\cap X|\ge 2t\}$, where $X$ is a partial permutation  of size $3t$. Note that for the two $t$-intersecting families, we have 
\[
\gamma_t(\mathcal{T}^*(X,t))=\gamma_t(\mathcal{T}(X,t))=\sum_{i=0}^{t}(-1)^i\binom{t}{i}(n-2t-i)!.
\]

We also determine the equality cases. Next, we characterize all $t$-intersecting families with maximum $t$-diversity.

\begin{theorem}\label{thm:all-extremal}
Let $t\ge1$, and let $\mathcal F\subset\Sigma_n$ be a $t$-intersecting family. Suppose that
\[
        \gamma_t(\mathcal F)=\sum_{i=0}^{t}(-1)^i\binom{t}{i}(n-2t-i)!,
\]
and $n\ge 130(t+1)(8t)^t$. Then $\mathcal F$ has one of the following structures.
\begin{enumerate}
    \item There is a partial permutation $X$ of size $3t$ such that
    \[
        \mathcal{T}^*(X,t)
        \subseteq
        \mathcal F
        \subseteq
        \mathcal{T}(X,t).
    \]

    \item $t=2$, and there is a Fano plane $(V,\mathcal L)$ whose point set $V$ is a $7$-point partial permutation such that
    \[
    \begin{aligned}
        &\{F\in\Sigma_n:
            F\supseteq V\setminus L
            \text{ and } |F\cap L|\le 1
            \text{ for some } L\in\mathcal L\}  \\
        &\hspace{2cm}\subseteq
        \mathcal F
        \subseteq
        \{F\in\Sigma_n:
            F\supseteq V\setminus L
            \text{ for some } L\in\mathcal L\}.
    \end{aligned}
    \]

\end{enumerate}
\end{theorem}

The proof has two main ingredients.  First, we apply the spread approximation lemma of Kupavskii and Zakharov to reduce a large $t$-intersecting family of permutations to a bounded family of partial permutations \cite{KUPAVSKII2024109653}.  A pseudo-sunflower reduction, based on a lemma of F\"uredi \cite{FUREDI1980282}, then reduces the problem to a $2t$-uniform $t$-intersecting family of partial permutations.  Second, we classify the resulting finite configurations.  The only possibilities are the complete $2t$-uniform family on a $3t$-point partial permutation and, in the exceptional case $t=2$, the family of complements of the lines of the Fano plane.

Moreover, it's natural to introduce another version's $t$-diversity, which measures the distance from a $t$-intersecting family to a $t$-star. 
In another word, this parameter measures the distance from such a family to a family with a $t$-transversal of size $t$, where a $t$-transversal is a set $t$-intersecting with every elements in $\mathcal{F}$. Here for convenience, $A$ $t$-intersects with $B$ means $|A\cap B|\ge t$, for $A,B\subseteq [n]$.

\begin{defn}\label{def:t-diversity}
Let $\mathcal F\subset\Sigma_n$.  The $t$-diversity of $\mathcal F$ is
\[
        \gamma_t^{\star}(\mathcal F)=
       \min_{\substack{T\text{ is a partial permutation}\\ |T|=t}}
        |\{F\in \mathcal{F}\colon T \not\subset F\}|
        =|F|-\Delta_t(\mathcal{F})
\]
where
\[
        \Delta_t(\mathcal{F})=\max_{\substack{T\text{ is a partial permutation}\\  |T|=t}}|\{F\in \mathcal{F}\colon T \subset F\}|.
\]
\end{defn}

Our next theorem gives the sharp upper bound and stability results for this version's parameter by using similar proof.

\begin{theorem}\label{thm:main2}
Let $t\geq1$, let $n\geq 6(t+4)^4$, and let $\mathcal{F}\subseteq \Sigma_n$ be $t$-intersecting.  Then
\[
  \gamma_t^{\star}(\mathcal{F})\leq t(n-t-2)(n-t-2)!.
\]
Equality holds if and only if there is a partial permutation $X$ of size $t+2$ such that
\[
  \{F\in\Sigma_n:|F\cap X|=t+1\}\subseteq \mathcal{F}\subseteq \{F\in\Sigma_n:|F\cap X|\geq t+1\}.
\]
\end{theorem}
For $t=1$, the right-hand side is $(n-2)!-(n-3)!=(n-3)(n-3)!$, so \Cref{thm:main2} also covers \Cref{thm:wang-xiao} with a weaker numerical threshold.

\section{Preliminary results}

We shall use the following spread approximation method introduced by Kupavskii and Zakharov.

 \begin{defn}
We say that a family $\mathcal{F} \subset 2^{[n]}$ is $r$-spread for some $r \ge 1$ if $\frac{|\mathcal{F}(X)|}{|\mathcal{F}|} \le r^{-|X|}$ for any $X \subset [n]$.
\end{defn}

 Let $0\leq p \leq 1$, we call a subset $W$ of $[n]$ $p$-random if each element of $[n]$ is included in $W$ with probability $p$ and independently of others.

 The following spread lemma is the core of the spread approximation method.
 It's a variant due to Stoeckl of the breakthrough result that was proved by Alweiss, Lovett, Wu and Zhang ${\cite{AlweissLovettWuZhang21}}$.

\begin{theorem}${\label{lem:spread-lemma}}$
(Spread Lemma ${\cite{Stoeckl22}}$). If for some $n, k, r, m \ge 1$ and $\delta > 0$ a family $\mathcal{F} \subset \binom{[n]}{\le k}$ is $r$-spread and $W$ is an $m\delta$-random subset of $[n]$, then

$$\Pr[\exists F \in \mathcal{F} : F \subset W] \ge 1 - \left( \frac{2}{\log_2(r\delta)} \right)^m k.$$
\end{theorem}

\begin{Observation}\label{ob:1}
Let $\mathcal F\subseteq 2^{[n]}$ be a
non-empty $r$-spread family, and let $R\subseteq[n]$ satisfy
$|R|\leq \alpha r$, where $0\leq\alpha<1$. 
Then $\mathcal F(\overline R)$ is $(1-\alpha)r$-spread.
\end{Observation}

\begin{proof}
Indeed,
\[
|\mathcal{F}(\bar{R})| \geq |\mathcal{F}| - \sum_{x \in R} |\mathcal{F}(x)| \geq |\mathcal{F}| - \frac{|R|}{r} |\mathcal{F}| \geq (1 - \alpha) |\mathcal{F}|.
\]

Then for any \( A \subset [n] \setminus R \) with \( A \neq \emptyset \),

\[
|\mathcal{F}(\bar{R})(A)| \leq |\mathcal{F}(A)| \leq r^{-|A|} |\mathcal{F}| \leq r^{-|A|}(1 - \alpha)^{-1} |\mathcal{F}(\bar{R})| \leq ((1 - \alpha)r)^{-|A|} |\mathcal{F}(\bar{R})|.
\]
\end{proof}

\begin{Observation}\label{ob:2}
Let $\mathcal F,\mathcal G\subseteq\binom{[n]}{\leq k}$ be
non-empty $r$-spread families. If $r>2^5\log_2(2k)$,
then there exist disjoint sets $F\in\mathcal F$ and $G\in\mathcal G$.
\end{Observation}

\begin{proof}
Color each $i\in [n]$ red or blue uniformly at random. Let $Y_1$ be the red set and $Y_2$ the blue set. Clearly both $Y_1$ and $Y_2$ can be viewed as $\frac{1}{2}$-random subset
of $[n]$. Let $m=\log_2(2k)$ and  $\delta=\frac{1}{2 \log_2(2k)}$. Then $r> 2^5\log_2(2k)$ implies
\[
\log_2(r\delta)>\log_2(2^4)= 4.
\]
Then by \Cref{lem:spread-lemma},
\[
Pr\left(\exists F\subset Y_1 \mbox{ for some }F\in \mathcal{F}\right)  \geq 1-\left(\frac{2}{\log_2(r \delta)}\right)^{\log_2(2k)} k >\frac{1}{2}
\]
and
\[
Pr\left(\exists G\subset Y_2 \mbox{ for some }G\in \mathcal{G}\right) \geq 1-\left(\frac{2}{\log_2(r \delta)}\right)^{\log_2(2k)} k >\frac{1}{2}.
\]
Thus, by the union bound, there exist disjoint sets $F\in \mathcal{F}$, $G\in \mathcal{G}$.
\end{proof}

For $i=1,2,\ldots,n$, let $X_i=\{x_{i1},x_{i2},\ldots,x_{in}\}$ and $X=X_1\cup X_2\cup \ldots\cup X_n$. 
A family $\mathcal{F}\subset \Sigma_n$ can be viewed as an $n$-partite 
$n$-graph on partite sets $X_1,X_2,\ldots,X_n$ with edge set 
$\{\{x_{1\sigma(1)},x_{2\sigma(2)},\ldots,x_{n\sigma(n)}\}:
\sigma\in\mathcal{F}\}$. From this point of view, 
$\mathcal{F}\subset \Sigma_n$ is also a family of $n$-subsets of $X=[n]^2$. Hence we can consider the $t$-diversity of $\Sigma_n$ as the same of ordinary set systems.

Let $\mathcal{F}\subseteq\Sigma_n$ and $r\geq1$.  We say that
a partial permutation $S$ is \emph{$r$-spread in $\mathcal{F}$} if
\[
  |\mathcal{F}(S)|\leq r^{-|S|}|\mathcal{F}|.
\]
Otherwise, we say that $S$ is \emph{not $r$-spread in $\mathcal{F}$}.
We call $S$ \emph{maximal non-$r$-spread in $\mathcal{F}$} if it is not
$r$-spread in $\mathcal{F}$, but every partial permutation $T$ with
$S\subsetneq T$ is $r$-spread in $\mathcal{F}$.

\begin{fact}\label{fact:nonempty-maximal-base}
Let $t\geq1$, and let $\mathcal{F}\subseteq\Sigma_n$ be
$t$-intersecting.  If $r>2^5\log_2(2n)$, then there exists a maximal non-$r$-spread partial permutation in $\mathcal{F}$.  Moreover, every such partial permutation is non-empty.
\end{fact}

\begin{proof}
Suppose that every partial permutation is $r$-spread in $\mathcal{F}$. By Observation~\ref{ob:2}, applying to two copies of $\mathcal{F}$ with $k=n$, gives disjoint $F_1,F_2\in\mathcal{F}$.  This is impossible because $\mathcal{F}$ is $t$-intersecting.  Hence some partial permutation
is not $r$-spread in $\mathcal{F}$, and finiteness gives a maximal non-$r$-spread
one.

Finally,
\[
  |\mathcal{F}(\varnothing)|=|\mathcal{F}|=r^0|\mathcal{F}|,
\]
so $\varnothing$ is $r$-spread in $\mathcal{F}$.  Therefore every partial
permutation that is not $r$-spread in $\mathcal{F}$ is non-empty.
\end{proof}

\begin{lemma}[Kupavskii--Zakharov {\cite{KUPAVSKII2024109653}}]\label{lem:spread-approximation}
Let \( n, t \geq 1 \) be integers. Let $\mathcal F\subseteq\Sigma_n$ be a
$t$-intersecting family, let $q$ be a positive integer and let $\tau\ge 1$, $x>1$. Put $r_0=\frac{n-q}{e}$ and
\[
  r_0\geq x\tau q
  \qquad\text{and}\qquad
  r_0>2^5\frac{x\tau}{x-1}\log_2(2n).
\]
Then there exists a $t$-intersecting family \(\mathcal{B}\) of partial permutations, each of size at most \(q\) and \(\mathcal{F}' \subset \mathcal{F}\) such that

\begin{enumerate}
\item for any \(B \in \mathcal{B}\) there is a family \(\mathcal{F}_B \subset \mathcal{F}\) such that \(\mathcal{F}_B(B)\) is $(r_0/\tau)$-spread;
\item \(\mathcal{F} \setminus \mathcal{F}' =\bigcup_{B\in\mathcal B}\mathcal F_B \);
\item $|\mathcal{F}'| \le \tau^{-q-1}|\Sigma_{n}| = \tau^{-q-1} n! $.
\end{enumerate}
\end{lemma}

\begin{proof}
Put $r=\frac{r_0}{\tau}.$

In the first stage, let $\mathcal{F}^1=\mathcal{F}$ and $\mathcal{B}=\varnothing$.  The second
hypothesis gives
\[
  r>2^5\frac{x}{x-1}\log_2(2n)>2^5\log_2(2n).
\]
Since $\mathcal{F}^1$ is $t$-intersecting, Fact~\ref{fact:nonempty-maximal-base} gives a non-empty maximal non-$r$-spread partial permutation $B_1$ in $\mathcal{F}^1$.  If $|B_1|>q$, then stop and put $\mathcal{F}'=\mathcal{F}^1$. Otherwise, put 
\[
  \mathcal{F}_{B_1}=\mathcal{F}^1[B_1],
  \qquad
  \mathcal{F}^2=\mathcal{F}^1\setminus\mathcal{F}_{B_1},
\]
and add $B_1$ to $\mathcal{B}$.

Continue in the same way.  At stage $i\geq2$, if
$\mathcal{F}^i=\varnothing$, stop and put $\mathcal{F}'=\varnothing$.  Otherwise
Fact~\ref{fact:nonempty-maximal-base} gives a maximal non-empty non-$r$-spread partial permutation $B_i$ in $\mathcal{F}^i$.
If $|B_i|>q$, stop and put $\mathcal{F}'=\mathcal{F}^i$.  If not, put
\[
  \mathcal{F}_{B_i}=\mathcal{F}^i[B_i],
  \qquad
  \mathcal{F}^{i+1}=\mathcal{F}^i\setminus\mathcal{F}_{B_i},
\]
and add $B_i$ to $\mathcal{B}$.

We first verify the spread conclusion.  Fix a recorded $B_i$ and let $\varnothing\neq A\subseteq X\setminus B_i$.  If $B_i\cup A$ is
a partial permutation, the maximal non-$r$-spread property of $B_i$
shows that $B_i\cup A$ is $r$-spread in $\mathcal{F}^i$, and hence
\[
  |\mathcal{F}^i(B_i\cup A)|
  \leq r^{-|B_i|-|A|}|\mathcal{F}^i|,
\]
whereas $B_i$ is not $r$-spread in $\mathcal{F}^i$, so
\[
  |\mathcal{F}^i(B_i)|>r^{-|B_i|}|\mathcal{F}^i|.
\]
Consequently,
\[
  \frac{|\mathcal{F}_{B_i}(B_i)(A)|}{|\mathcal{F}_{B_i}(B_i)|}
  =\frac{|\mathcal{F}^i(B_i\cup A)|}{|\mathcal{F}^i(B_i)|}
  <r^{-|A|}.
\]
If $B_i\cup A$ is not a partial permutation, then
$\mathcal{F}_{B_i}(B_i)(A)=\varnothing$.  Thus
$\mathcal{F}_{B_i}(B_i)$ is $r=(r_0/\tau)$-spread.

The peeling construction gives
\[
  \mathcal{F}\setminus\mathcal{F}'
  =\mathop{\dot\bigcup}_{B\in\mathcal{B}}\mathcal{F}_B.
\]
If $\mathcal{F}'=\varnothing$, the remainder bound is immediate.  Otherwise
the procedure stopped with a maximal non-$r$-spread partial permutation
$B_m$ in $\mathcal{F}'$ having size $a=|B_m|\ge q+1$.  Thus
\[
  |\mathcal{F}'(B_m)|>r^{-a}|\mathcal{F}'|,
\]
and $|\mathcal{F}'(B_m)|\leq(n-a)!$.  Therefore
\begin{align*}
  |\mathcal{F}'|
  &<r^a|\mathcal{F}'(B_m)|
   \leq\left(\frac{n-q}{e\tau}\right)^a(n-a)! \\
  &\leq\tau^{-a}n!
   \leq\tau^{-q-1}n!.
\end{align*}
Here we used the standard estimate
$n!/(n-a)!\geq(n/e)^a$.

It remains to show that $\mathcal B$ is $t$-intersecting.
Suppose not and take any (not necessarily distinct) \( B_1, B_2 \in \mathcal{B} \) and assume that \( |B_1 \cap B_2| < t \). 
Let
\[
  \mathcal H_{B_1}:=\mathcal F_{B_1}(B_1),
  \qquad
  \mathcal H_{B_2}:=\mathcal F_{B_2}(B_2),
\]
and delete from each residual family the points belonging to the other base:
\[
  \mathcal G_{B_1}:=\mathcal H_{B_1}(\overline{B_2\setminus B_1}),
  \qquad
  \mathcal G_{B_2}:=\mathcal H_{B_2}(\overline{B_1\setminus B_2}).
\]
Both $\mathcal H_{B_1}$ and $\mathcal H_{B_2}$ are $(r_0/\tau)$-spread, and
\[
  |B_2\setminus B_1|,|B_1\setminus B_2|
  \leq q
  \leq \frac1x\frac{r_0}{\tau}.
\]
By Observation~\ref{ob:1}, with $\alpha=1/x$, we get that both
$\mathcal G_{B_1}$ and $\mathcal G_{B_2}$ are $r'$-spread, where
\[
  r':=\left(1-\frac1x\right)\frac{r_0}{\tau}
  =\frac{x-1}{x\tau}r_0
  >2^5\log_2(2n).
\]
Moreover, both are families of sets of size at most $n$ on the common ground
set $([n]\times[n])\setminus(B_1\cup B_2)$. By Observation~\ref{ob:2}, there are
disjoint sets $G_{1}\in\mathcal G_{B_1}$ and $G_{2}\in\mathcal G_{B_2}$. Then
$G_{1}\cup B_1,G_{2}\cup B_2\in\mathcal F$, whereas
\[
  |(G_{1}\cup B_1)\cap(G_{2}\cup B_2)|=|B_1\cap B_2|<t,
\]
contradicting the $t$-intersection property of $\mathcal F$. Hence
$\mathcal B$ is $t$-intersecting, and the proof is complete.

\end{proof}

\section{The upper bound: Proof of \Cref{thm:main-bound}}\label{sec:upper}

We begin with two elementary counting facts.  If $S$ is a partial permutation of size $a$, then $|\Sigma_n[S]|=(n-a)!$. For convenience, let
\[
        D_t(n)=\sum_{i=0}^{t}(-1)^i\binom{t}{i}(n-2t-i)!.
\]
The following refinement is used repeatedly.

\begin{lemma}\label{lem:counting}
For $n\ge 3t$, let $S$ be a partial permutation of size $2t$, and let $T$ be a $t$-element partial permutation disjoint from $S$. 
\begin{enumerate}[(i)]
    \item If $S\cup T$ is a partial permutation, then
    \[
            |\Sigma_n[S,\overline T]|=D_t(n).
    \]
    \item In general,
    \[
            |\Sigma_n[S,\overline T]|\ge D_t(n).
    \]
\end{enumerate}
\end{lemma}

\begin{proof}
After fixing $S$, the remaining choices form a permutation on $n-2t$ rows and columns.  If $S\cup T$ is a partial permutation, the condition of avoiding $T$ is the condition of avoiding $t$ specified positions in this remaining permutation.  Inclusion--exclusion gives
\[
        |\Sigma_n[S,\overline T]|
        =\sum_{i=0}^{t}(-1)^i\binom{t}{i}(n-2t-i)!.
\]
In the general case, some elements of $T$ are incompatible with $S$ and hence are automatically avoided by every permutation containing $S$.  The compatible part of $T$ has size at most $t$ and can be extended, in the remaining rows and columns, to a partial permutation of size $t$.  Avoiding fewer specified positions can only increase the number of permutations, so the same expression gives a lower bound.
\end{proof}

\begin{lemma}\label{lem:large-fibre}
Let $t\ge1$ and $n\ge 2t+6$.  Let $\mathcal G\subseteq\Sigma_n$ be a $t$-intersecting family, and let $B$ be a partial permutation of size $2t$.  If
\[
        |\mathcal G[B]|\ge \left(1-\frac1e+0.001\right)(n-2t)!,
\]
then every $G\in\mathcal G$ satisfies $|G\cap B|\ge t$.
\end{lemma}

\begin{proof}
Suppose that some $G_0\in\mathcal G$ satisfies $|G_0\cap B|\le t-1$.  Every member of $\mathcal G[B]$ contains $B$ and must $t$-intersect with $G_0$.  Hence every member of $\mathcal G[B]$ must contain at least one point of $G_0\setminus B$ that is compatible with $B$.

After relabelling, assume $B=\{(1,1),\ldots,(2t,2t)\}$.  Once $B$ is fixed, the remaining choices form a permutation on $m=n-2t$ rows and columns.  The compatible part of $G_0\setminus B$ is a partial permutation in this remaining $m\times m$ board, of size at most $m$.  The number of permutations on $m$ points meeting a fixed partial permutation is at most the number meeting a full permutation, namely
\[
        m!\sum_{i=1}^{m}\frac{(-1)^{i-1}}{i!}
        \le
        \left(1-\frac1e+0.0005\right)m!
\]
for all $m\ge6$.  We have $m=n-2t\ge6$, hence
\[
        |\mathcal G[B]|
        \le
        \left(1-\frac1e+0.0005\right)(n-2t)!
        <
        \left(1-\frac1e+0.001\right)(n-2t)!,
\]
contradicting the assumed lower bound on $|\mathcal G[B]|$.
\end{proof}

We collect below the numerical consequences of the lower bound on
$n$ that will be used later. These inequalities allow us to apply
\cref{lem:spread-approximation} and justify the subsequent peeling and rescaling estimates.
\begin{lemma}\label{Counting}
Let $n\ge 130(t+1)(8t)^t, t\ge 1$ and $x=5.2,\tau=1.3,r_0=\frac{n-q}{e},q=\lceil (2t+1)\log_{\tau} n\rceil$, then
\begin{enumerate}[(i)]
    \item $r_0\geq x\tau q$ and $r_0>2^5\frac{x\tau}{x-1}\log_2(2n)$,
 
    \item $2^{-5}(n-2t)\ge 4\binom{2t+1}{t}(2t)^{t+1}+\frac{10}{13}$ and $n\ge(4e+1)q$.
\end{enumerate}
\end{lemma}
\begin{proof}
We first prove $n>20q$ needed for both inequalities.

Since $f(u):=0.65\sqrt u-\log_{1.3}u$ is increasing and positive for
$u\ge2080$, we have
\[
  \log_{\tau}n<0.65\sqrt n.                              
\]

For $t=1$, the estimate $n>20q$ follows directly from the above inequality.  
For $t\ge2$,
using $(8t)^t\ge64t^2$ and
$2t+1\le2.5t$, we have
\[
 n\ge24960t^2>2500t^2\ge\bigl(20(2t+1)\bigr)^2.
\]
Thus,
\[
 20q\le 20((2t+1)\log_{\tau}n+1)
   <0.65n+20<n.
\]
Consequently,
\[
  r_0=\frac{n-q}{e}
      \ge\frac{19q}{e}
      >6.76q=x\tau q.
\]

Since $g(u):=u-152\log_2(2u)$ is increasing and positive for
$u\ge2080$, we have \[152\log_2(2n)<n.\]
Thus, 
\[
  2^5\frac{x\tau}{x-1}\log_2(2n)
     <0.34n<\frac{0.95n}{e}\le r_0.
\]
This proves (i).

For $n>20q$, we can easily obtain $n>(4e+1)q$. Using the elementary estimate
\[
  \binom{2t+1}{t}\le\frac{t+1}{2t}4^t
\]
and \((t+1)(8t)^t\ge16t\), for \(n\ge130(t+1)(8t)^t\), we have
\[
\begin{aligned}
  2^{-5}(n-2t)
  &\ge
  4(t+1)(8t)^t
  +\frac{2(t+1)(8t)^t-2t}{32}\\
  &\ge
  4\binom{2t+1}{t}(2t)^{t+1}
  +\frac{30}{32}\\
  &>
  4\binom{2t+1}{t}(2t)^{t+1}
  +\frac{10}{13}.
\end{aligned}
\]
This proves (ii).  
\end{proof}

We now recall the notion of a pseudo-sunflower, introduced in \cite{Frankl22}.

\begin{defn}
Sets $F_0,F_1,\ldots,F_s\subseteq X$ form a pseudo-sunflower of size $s+1$ with center $C$ if $C\subsetneq F_0$ and the sets
\[
        F_0\setminus C,
        F_1\setminus C,
        \ldots,
        F_s\setminus C
\]
are pairwise disjoint.
\end{defn}

The following result was proved by F\"uredi \cite{FUREDI1980282}; see also \cite{Frankl22}.

\begin{theorem}\label{thm:pseudo-sunflower-free}
Let $\mathcal H$ be a $k$-uniform $t$-intersecting family without a pseudo-sunflower of size $s+1$.  Then
\[
        |\mathcal H|\le s^k-\sum_{0 \le j < t} \binom{k}{j}(s-1)^{k-j}.
\]
\end{theorem}

\begin{proof}[Proof of \Cref{thm:main-bound}]
We need a simple pruning operation.  Let $\mathcal H$ be a $t$-intersecting family of partial permutations, all of size at most $s$.  If $\mathcal H$ contains a pseudo-sunflower $F_0,F_1,\ldots,F_s$ with center $C\subsetneq F_0$, then every $H\in\mathcal H$ satisfies $|H\cap C|\ge t$.  Indeed, otherwise $H$ would have to meet each of the $s+1$ pairwise disjoint petals $F_i\setminus C$, impossible since $|H|\le s$.  Therefore, after removing all members of $\mathcal H$ that contain $C$ and adding the center $C$, the resulting family remains $t$-intersecting.  Let $\mathcal R(\mathcal H,s)$ denote the final family.  Then $\mathcal R(\mathcal H,s)$ is $t$-intersecting, consists of partial permutations of size at most $s$, and contains no pseudo-sunflower of size $s+1$.  Moreover, every original member of $\mathcal H$ contains some member of the final family or of one of the layers removed during the iteration.

Next, we apply the \cref{lem:spread-approximation}  to $\mathcal F$ with $ x=5.2,\tau=1.3,r_0=\frac{n-q}{e},q=\lceil (2t+1)\log_{\tau} n\rceil$, where all the conditions of the lemma are satisfied by \cref{Counting} $(i)$, obtaining that there exists a $t$-intersecting family $\mathcal B$ of partial permutations, each of size at most $q$, and $\mathcal F'\subseteq\mathcal F$ such that
\[
        \mathcal F=\mathcal F[\mathcal B]\cup\mathcal F'
        \qquad and\qquad
        |\mathcal F'|\le\tau^{-1} n^{-(2t+1)}n!\le \frac{10}{13} (n-2t-1)!.
\]

Set $\mathcal P_q=\mathcal R(\mathcal B,q)$ and let $\mathcal A_q$ be the subfamily of all members of $\mathcal P_q$ of size $q$.  For $i=q-1,q-2,\ldots,2t+1$, after constructing $\mathcal P_{i+1}$ and $\mathcal A_{i+1}$, define
\[
        \mathcal P_i=\mathcal R(\mathcal P_{i+1}\setminus\mathcal A_{i+1},i),
        \qquad
        \mathcal A_i=\{A\in\mathcal P_i: |A|=i\}.
\]
Finally put
\[
        \mathcal A=\mathcal P_{2t+1}\setminus\mathcal A_{2t+1}.
\]
Then $\mathcal A$ is a $t$-intersecting family of partial permutations, each of size at most $2t$.  Also, for each $i\in\{2t+1,\ldots,q\}$, the family $\mathcal A_i$ is $i$-uniform $t$-intersecting and contains no pseudo-sunflower of size $i+1$, so \Cref{thm:pseudo-sunflower-free} gives
\[
        |\mathcal A_i|\le i^i-\sum_{0 \le j < t} \binom{i}{j}(i-1)^{i-j}=\sum_{t \le j \le i} \binom{i}{j}(i-1)^{i-j}.
\]
Since $j\ge t\ge1$,we have 
\[
\frac{\binom{i}{j+1}(i-1)^{i-j-1}}{\binom{i}{j}(i-1)^{i-j}}=\frac{i-j}{j+1}\cdot\frac{1}{i-1}\le\frac{i-t}{t+1}\cdot\frac{1}{i-1}\le \frac{1}{2}.
\]
Hence   $|\mathcal A_i|\le2\binom{i}{t}(i-1)^{i-t}.$

By construction, every $B\in\mathcal B$ contains some member of $\mathcal A\cup\mathcal A_{2t+1}\cup\cdots\cup\mathcal A_q$.  Hence
\[
        \mathcal F[\mathcal B]
        \subseteq
        \mathcal F[\mathcal A]\cup
        \bigcup_{i=2t+1}^{q}\mathcal F[\mathcal A_i].
\]
Let
\[
        \mathcal F''=\bigcup_{i=2t+1}^{q}\mathcal F[\mathcal A_i]\cup\mathcal F'.
\]
Since for all $n\ge (4e+1)q, 2t\le i\le q$, we have
\begin{equation}
\frac{2\binom{i+1}{t}(i)^{i+1-t}(n-i-1)!}
{2\binom{i}{t}(i-1)^{i-t}(n-i)!}
=\frac{i+1}{i+1-t}\cdot i
\left(1+\frac{1}{i-1}\right)^{i-t}\cdot\frac{1}{n-i}
\le \frac{2ei}{n-i}
\le\frac{2eq}{n-q}
\le\frac12.
\tag{$\ast$}\label{eq:star}
\end{equation}
Therefore, by \cref{Counting}$(ii)$, for $n\ge 130(t+1)(8t)^t$
\[
\begin{aligned}
        |\mathcal F''|
        &\le \sum_{i=2t+1}^{q}|\mathcal A_i|(n-i)!+|\mathcal F'| \\
        &\le 4\binom{2t+1}{t}(2t)^{t+1}(n-2t-1)!+\frac{10}{13}(n-2t-1)! \\
        &=(4\binom{2t+1}{t}(2t)^{t+1}+\frac{10}{13})(n-2t-1)! \\
        &\le 2^{-5}(n-2t)!.
\end{aligned}
\]

If there exists a $t$-partial permutation $T$ meeting every member of $\mathcal A$, then $\mathcal F[\mathcal A](\overline T)=\emptyset$, so
\[
        \gamma_t(\mathcal F)
        \le |\mathcal F(\overline T)|
        \le |\mathcal F''|
        <2^{-5}(n-2t)!<D_t(n)
\]
and \Cref{thm:main-bound} follows.  Hence we may assume from now on that
\begin{equation*}
        \text{for every $t$-partial permutation $T$, there exists $A\in\mathcal A$ such that $A\cap T=\emptyset$.}
\end{equation*}
Let us recall the transversal number $\tau(\mathcal F):=\min\bigl\{|T|:\ T\text{ is a transversal of }\mathcal F\bigr\}$, so the condition above is equivalent to $\tau(\mathcal A)>t$.

\begin{claim}\label{claim:2t-uniform}
The family $\mathcal A$ is $2t$-uniform.
\end{claim}

\begin{proof}
Every member of $\mathcal A$ has size at most $2t$.  Since $\mathcal{A}$ is $t$-intersecting, no member has size less than $t$. Now suppose that some $A\in\mathcal A$ has size at most $2t-1$, and choose a $t$-subset $T\subseteq A$.  For any $A'\in\mathcal A$, the $t$-intersection of $A'$ with $A$ gives
\[
        |A'\cap T|+|A'\cap(A\setminus T)|\ge t.
\]
But $|A\setminus T|\le t-1$, so $A'\cap T\ne\emptyset$.  This contradicts $\tau(\mathcal A)>t$.  Thus every member has size exactly $2t$.
\end{proof}

Choose $A\in\mathcal A$ and a $t$-subset $X\subset A$.  Write
\[
        A=X\sq Y,
        \qquad |Y|=t.
\]
By $\tau(\mathcal A)>t$, there is $B\in\mathcal A$ disjoint from $X$.  Since $|A\cap B|\ge t$ and $B\cap X=\emptyset$, we have $A\cap B=Y$.  Write
\[
        B=Y\sq Z,
\]
where $Z$ is a $t$-set disjoint from $A$.  Again by $\tau(\mathcal A)>t$, there is $C\in\mathcal A$ disjoint from $Y$.  The $t$-intersection property gives $C\cap A=X$ and $C\cap B=Z$, so $C\supseteq X\sq Z$.  Since $|C|=2t$, we have
\[
        C=X\sq Z.
\]
Moreover, the same argument shows that $C=X\sq Z$ is the unique member of $\mathcal A$ disjoint from $Y$.  Notice that $X\sq Y\sq Z$ is a partial permutation, because each of $X\sq Y$, $Y\sq Z$, and $X\sq Z$ is a partial permutation.

Assume first that
\begin{equation}\label{eq:both-large}
        |\mathcal F[X\sq Y]|\ge\left(1-\frac1e+0.001\right)(n-2t)!
        \quad\text{and}\quad
        |\mathcal F[X\sq Z]|\ge\left(1-\frac1e+0.001\right)(n-2t)! .
\end{equation}
By Lemma~\ref{lem:large-fibre}, every $F\in\mathcal F$ $t$-intersects with both $X\sq Y$ and $X\sq Z$.  If $F\in\mathcal F(\overline X)$, then
\[
        F\cap(X\sq Y)=Y,
        \qquad
        F\cap(X\sq Z)=Z.
\]
Hence
\[
        F\cap(X\sq Y\sq Z)=Y\sq Z.
\]
By Lemma~\ref{lem:counting},
\[
\begin{aligned}
        \gamma_t(\mathcal F)
        &\le |\mathcal F(\overline X)| \\
        &\le |\{F\in\Sigma_n:F\cap(X\sq Y\sq Z)=Y\sq Z\}| \\
        &=D_t(n).
\end{aligned}
\]

It remains to handle the case where at least one inequality in \eqref{eq:both-large} fails.  By symmetry, suppose
\[
        |\mathcal F[X\sq Y]|<\left(1-\frac1e+0.001\right)(n-2t)!.
\]
Let $F\in\mathcal F[\mathcal A](\overline Z)$.  Then $F$ contains some $A'\in\mathcal A$.  Since $A'$ $t$-intersects with $Y\sq Z$ and $X\sq Z$, while $F$ is disjoint from $Z$, we must have $Y\subseteq F$ and $X\subseteq F$.  Hence $F\supseteq X\sq Y$, and therefore
\[
        \mathcal F[\mathcal A](\overline Z)\subseteq \mathcal F[X\sq Y].
\]
It follows that
\[
\begin{aligned}
        \gamma_t(\mathcal F)
        &\le |\mathcal F(\overline Z)| \\
        &\le |\mathcal F[\mathcal A](\overline Z)|+|\mathcal F''| \\
        &\le |\mathcal F[X\sq Y]|+|\mathcal F''| \\
        &<\left(1-\frac1e+0.001+2^{-5}\right)(n-2t)!.
\end{aligned}
\]
The last expression is smaller than $D_t(n)$ in the present range: indeed, $1-1/e+0.001+2^{-5}<\frac23$, while $n\ge 130(t+1)(8t)^t$ leads to $n-2t>3t$, so by the union bound
\[
        D_t(n)\ge (n-2t)!-t(n-2t-1)!>\frac23 (n-2t)!.
\]
This proves \Cref{thm:main-bound}.
\end{proof}

\section{The extremal families: Proof of \Cref{thm:all-extremal}}\label{sec:all}
We first show that maximal extremal families are exactly the families generated by finite $2t$-uniform configurations with $\tau(\mathcal{A})>t$.

\begin{lemma}\label{lem:F=sigma-A}
Let $t\ge1$ and $n\ge 130(t+1)(8t)^t$.  The following two statements hold.
\begin{enumerate}[(i)]
    \item If $\mathcal A$ is a $2t$-uniform $t$-intersecting family of partial permutations such that $\mathcal A(\overline T)\ne\emptyset$ for every $t$-partial permutation $T$, then $\Sigma_n[\mathcal A]$ is a maximal $t$-intersecting family and
    \[
            \gamma_t(\Sigma_n[\mathcal A])=D_t(n).
    \]
    \item Conversely, if $\mathcal F\subseteq\Sigma_n$ is a maximal $t$-intersecting family with $\gamma_t(\mathcal F)=D_t(n)$, then
    \[
            \mathcal F=\Sigma_n[\mathcal A]
    \]
    for some $2t$-uniform $t$-intersecting family $\mathcal A$ of partial permutations satisfying $\mathcal A(\overline T)\ne\emptyset$ for every $t$-partial permutation $T$.
\end{enumerate}
\end{lemma}

\begin{proof}
We prove (i).  Let $\mathcal F=\Sigma_n[\mathcal A]$.  Since $\mathcal A$ is $t$-intersecting, $\mathcal F$ is $t$-intersecting.  For any $t$-partial permutation $T$, choose $A_T\in\mathcal A$ disjoint from $T$.  By Lemma~\ref{lem:counting},
\[
        |\mathcal F(\overline T)|
        \ge |\Sigma_n[A_T,\overline T]|
        \ge D_t(n).
\]
Thus $\gamma_t(\mathcal F)\ge D_t(n)$.

Choose $A\in\mathcal A$ and a $t$-subset $X\subset A$.  Write $A=X\sq Y$.  As in the proof of \Cref{thm:main-bound}, using the avoidance property inside $\mathcal A$, there are members
\[
        X\sq Y,
        \qquad
        Y\sq Z,
        \qquad
        X\sq Z
\]
of $\mathcal A$.  If $F\in\Sigma_n[\mathcal A](\overline X)$, then $F$ contains some $A'\in\mathcal A$.  Since $A'$ $t$-intersects with both $X\sq Y$ and $X\sq Z$, and $F$ is disjoint from $X$, it follows that $Y\sq Z\subseteq F$.  Hence
\[
        |\mathcal F(\overline X)|
        \le
        |\{F\in\Sigma_n:F\cap(X\sq Y\sq Z)=Y\sq Z\}|
        =D_t(n).
\]
So $\gamma_t(\mathcal F)=D_t(n)$.

It remains to prove maximality.  Suppose that a permutation $P\notin\Sigma_n[\mathcal A]$ $t$-intersects with every member of $\Sigma_n[\mathcal A]$.  Then $\mathcal G=\Sigma_n[\mathcal A]\cup\{P\}$ is $t$-intersecting.  For each $A\in\mathcal A$, the fibre $\mathcal G[A]$ contains all $(n-2t)!$ permutations containing $A$.  By Lemma~\ref{lem:large-fibre}, $P$ $t$-intersects with every $A\in\mathcal A$.

Choose $A_0\in\mathcal A$ minimizing $|P\cap A_0|$.  Since $P$ $t$-intersects with every member of $\mathcal A$, and since $P$ contains no member of $\mathcal A$, we have
\[
        t\le |P\cap A_0|\le 2t-1.
\]
Choose $X\subseteq P\cap A_0$ with $|X|=t$, and write $A_0=X\sq Y$.  As above, there are $A_1=Y\sq Z$ and $A_2=X\sq Z$ in $\mathcal A$.  Since $P\supseteq X$ but $P\not\supseteq A_2$, we have $|P\cap Z|\le t-1$.  Therefore
\[
        |P\cap A_1|
        =|P\cap Y|+|P\cap Z|
        \le |P\cap A_0|-t+(t-1)
        =|P\cap A_0|-1,
\]
contradicting the choice of $A_0$.  Hence $\Sigma_n[\mathcal A]$ is maximal.

We now prove (ii).  Let $\mathcal F$ be maximal and satisfy $\gamma_t(\mathcal F)=D_t(n)$.  Let $\mathcal A$ be the core family constructed in the proof of \Cref{thm:main-bound}.  If some $t$-partial permutation $T$ meets every member of $\mathcal A$, then the first case in the proof of \Cref{thm:main-bound} gives $\gamma_t(\mathcal F)<D_t(n)$, a contradiction.  Thus $\mathcal A(\overline T)\ne\emptyset$ for every $t$-partial permutation $T$.  By Claim~\ref{claim:2t-uniform}, $\mathcal A$ is $2t$-uniform and $t$-intersecting.

We claim that every member of $\mathcal F$ $t$-intersects with every member of $\mathcal A$.  Suppose not.  Then for some $A=X\sq Y\in\mathcal A$, the fibre $\mathcal F[A]$ must have size less than $(1-1/e+0.001)(n-2t)!$; otherwise Lemma~\ref{lem:large-fibre} would force every member of $\mathcal F$ to $t$-intersect with $A$.  Applying the proof of \Cref{thm:main-bound} to the triple $X\sq Y$, $Y\sq Z$, $X\sq Z$ gives
\[
        \gamma_t(\mathcal F)<D_t(n),
\]
a contradiction.  Therefore every $F\in\mathcal F$ $t$-intersects with every $A\in\mathcal A$.

The minimization argument from the preceding paragraph shows more generally that any permutation $P\in\Sigma_n$ which $t$-intersects with every member of $\mathcal A$ must contain some member of $\mathcal A$.  Hence $\mathcal F\subseteq\Sigma_n[\mathcal A]$.  Since $\Sigma_n[\mathcal A]$ is $t$-intersecting by (i), the maximality of $\mathcal F$ implies
\[
        \mathcal F=\Sigma_n[\mathcal A].
\]
This proves (ii).
\end{proof}

It remains to classify the possible finite configurations $\mathcal A$.  We first show that their union is itself a partial permutation.

\begin{prop}\label{prop:union-partial}
Let $\mathcal A$ be a $2t$-uniform $t$-intersecting family of partial permutations such that $\mathcal A(\overline T)\ne\emptyset$ for every $t$-partial permutation $T$.  Then, for any two members $A,B\in\mathcal A$, there exists $C\in\mathcal A$ such that
\[
        A\triangle B\subseteq C.
\]
Moreover, $\bigcup_{A\in\mathcal A}A$ is a partial permutation.
\end{prop}

\begin{proof}
Choose a partial permutation $T\in \binom{A\cap B}{t}$, there exists $C\in \mathcal{A}$ such that $T\cap C=\emptyset$. Then $C\cap A=A\setminus T\supseteq A\setminus B$ and $C\cap B=B\setminus T\supseteq B\setminus A$. Hence, $A\triangle B\subseteq C$.

Suppose $\bigcup_{A\in\mathcal A}A$ is not a partial permutation.  Then there are two distinct pairs $a=(r,c)$ and $b=(r',c')$ in this union with $r=r'$ or $c=c'$.  Choose $A_1,A_2\in\mathcal A$ with $a\in A_1$ and $b\in A_2$.  Since no partial permutation can contain both $a$ and $b$, we have $a\in A_1\setminus A_2$ and $b\in A_2\setminus A_1$.  By the first part, there exists $A_3\in\mathcal A$ with $A_1\triangle A_2\subseteq A_3$.  Then $A_3$ contains both conflicting pairs $a$ and $b$, contradiction.  Therefore the union is a partial permutation.
\end{proof}

\begin{rmk}
For completeness, we record the elementary connectivity property of the Kneser graph that will be used in the proof below.

Let $Y$ be a $2t$-set. The Kneser graph $KG(2t,t-1)$ is the graph
with vertex set $\binom{Y}{t-1}$, where two vertices are adjacent
if and only if they are disjoint.

Indeed, take any $P,Q\in\binom{Y}{t-1}$. Write
\[
  P\setminus Q=\{p_1,\ldots,p_m\}
  \qquad and \qquad
  Q\setminus P=\{q_1,\ldots,q_m\}
\]
and, for $0\le i\le m$, let $P_{i}:=\bigl(P\setminus\{p_1,\ldots,p_i\}\bigr)\cup\{q_1,\ldots,q_i\}.$
Then $P_0=P$, $P_m=Q$ and the two sets $P_{i-1},P_i$
differ in exactly one element, where we have $|P_{i-1}\cup P_i|=t$.

Since $|Y|=2t$, the set $Y\setminus(P_{i-1}\cup P_i)$ has size $t$. We choose
\[
  H_i\in
  \binom{\,Y\setminus(P_{i-1}\cup P_i)\,}{t-1}.
\]
Then $H_i$ is disjoint from both $P_{i-1}$ and $P_i$, so
$P_0,H_1,P_1,H_2,\ldots,H_m,P_m$
is a path from $P$ to $Q$. Thus $KG(2t,t-1)$ is connected.
\end{rmk}

We now prove the finite structure theorem used in the maximal classification.

\begin{prop}\label{prop:A-structure}
Let $t\ge 2$, and let $\mathcal A\subseteq\binom{X}{2t}$ be a non-empty $t$-intersecting family with $X=\bigcup_{A\in\mathcal A}A$.  Suppose that for every $T\in\binom{X}{t}$, there exists $A\in\mathcal A$ such that $A\cap T=\emptyset$.  Then the following classification holds.
\begin{enumerate}
    \item If there are two distinct sets  $A_1,A_2\in\mathcal{A}$ with $|A_1\cap A_2|>t$, then $|X|=3t$ and $\mathcal A=\binom{X}{2t}$.
    \item If any two distinct $A_1,A_2\in \mathcal{A}$ have intersection of size exactly $t$, then $t=2,|X|=7$ and $\mathcal A$ is the family of complements of the lines of the Fano plane.   If $(X,\mathcal L)$ is the Fano plane, then
    \[
            \mathcal A=\{X\setminus L:L\in\mathcal L\}.
    \]
\end{enumerate}
\end{prop}

\begin{proof}
Fix $Y\in\mathcal{A}$ and set $U=X\setminus Y$. For
$T\in\binom Yt$, choose a member of $\mathcal{A}$ disjoint from $T$.
Its intersection with $Y$ has size at least $t$ and is contained in
$Y\setminus T$, so it has the form
\[
  (Y\setminus T)\dot\cup R_T,
  \qquad R_T\in\binom Ut.
\]
Doing the same for $Y\setminus T$ and comparing the two resulting
members gives
\begin{equation}\label{eq:complement}
  R_T=R_{Y\setminus T}.
\end{equation}
The same comparison also proves uniqueness. Consequently,
\begin{equation}\label{eq:canonical}
  T\dot\cup R_T\in\mathcal{A},\qquad
  (Y\setminus T)\dot\cup R_T\in\mathcal{A}
  \quad\text{for every }T\in\binom Yt.
\end{equation}

We record the only local structural fact that will be needed. Fix
$Q\in\binom{Y}{t-1}$ and, for $y\in Y\setminus Q$, abbreviate
$R_y=R_{Q\cup\{y\}}$. If $y\ne z$, then comparison in
\eqref{eq:canonical} of
\[
  (Q\cup\{y\})\dot\cup R_y
  \quad\text{and}\quad
  \bigl(Y\setminus(Q\cup\{z\})\bigr)\dot\cup R_z
\]
gives
\begin{equation}\label{eq:local}
  |R_y\cap R_z|\ge t-1.
\end{equation}
We claim that there are exactly two possibilities:
\begin{equation}\label{eq:dichotomy}
  \text{either all the $R_y$ are equal , or }
  \{R_y:y\in Y\setminus Q\}=\binom{W_Q}{t}
\end{equation}
for some $(t+1)$-set $W_Q\subseteq U$.

Indeed, let $L_Q=\bigcap_{y\in Y\setminus Q}R_y$. Suppose first that
$L_Q\ne \varnothing$ and fix $u\in L_Q$, choose $B\in\mathcal{A}$ disjoint from $Q\cup\{u\}$.
Since $B$ meets $Y$ in at least $t$ points and avoids $Q$, we have
$t\le |B\cap Y|\le t+1$. If $|B\cap Y|=t$, then for
$\{y_0\}= Y\setminus (Q\cup (B\cap Y))$, by the uniqueness of $R_{y_0}$,
\[
  B=\bigl(Y\setminus(Q\cup\{y_0\})\bigr)\dot\cup R_{y_0},
\]
contradicting $u\in R_{y_0}$. Hence $B\cap Y=Y\setminus Q$.
Let $B=(Y\setminus Q)\dot\cup S$, where $|S|=t-1$ and $S\subseteq U$. Since $B$ $t$-intersects with $(Q\cup\{y\})\dot\cup R_y$, we have $S\subseteq R_y$ for every $y\in Y\setminus Q$.
Moreover, $u\notin S$ while $u\in R_y$ for every $y$. Thus
$R_y=S\cup\{u\}$ for every $y$, proving the first case in
\eqref{eq:dichotomy}.

If $L_Q=\varnothing$, fix one of the $R_y$, say $R_0$. By
\eqref{eq:local}, every other $R_y$ differs from $R_0$ in at most one
element. $L_Q=\varnothing$ means every element in $R_0$ must be excluded by some $R_y$ which forces the other $t$ indexed sets to be the form
$R_0\setminus\{a\}\cup\{b_a\}$ for $a\in R_0$.
The bound \eqref{eq:local} forces all $b_a$ to be one common
element $b\notin R_0$. Hence the $R_y$ are precisely the $t$-subsets
of $R_0\cup\{b\}$, proving the second case in
\eqref{eq:dichotomy}.

We shall also use the following result. Suppose the
first case of \eqref{eq:dichotomy} holds for $Q$, with common
value $R$. If $P\in\binom{Y}{t-1}$ is disjoint from $Q$, write
\[
  Y=Q\dot\cup P\dot\cup\{a,b\}.
\]
By \eqref{eq:complement},
\[
  R_{P\cup\{a\}}=R_{Q\cup\{b\}}=R,
  \qquad
  R_{P\cup\{b\}}=R_{Q\cup\{a\}}=R.
\]
The sets in the second case of \eqref{eq:dichotomy} are all
distinct, so the first case also holds for $P$, with the same
value $R$. By the connectivity of the Kneser graph $KG(2t,t-1)$, we have
\begin{equation}\label{eq:constant}
  R_T=R\qquad\text{for every }T\in\binom Yt
\end{equation}
as soon as the first case holds for some $Q$.

Suppose first that there exist two distinct members of $\mathcal{A}$ intersecting larger than $t$. We take $Y,Y'\in\mathcal{A}$ so that
$|Y\cap Y'|>t$, hence $|Y\setminus Y'|\le t-1$. Choose $Q\in\binom{Y}{t-1}$ containing
$Y\setminus Y'$. For every $y\in Y\setminus Q$, comparison of
$(Q\cup\{y\})\dot\cup R_y$ with $Y'$ gives
\[
  |R_y\cap(Y'\setminus Y)|
  \ge 2t-|Y\cap Y'|=|Y'\setminus Y|.
\]
Thus $Y'\setminus Y$ is contained in every $R_y$. Since $Y'\setminus Y\neq\varnothing$ and $ Y'\setminus Y\subseteq L_Q$, so the first case in \eqref{eq:dichotomy} holds, and hence \eqref{eq:constant} holds.

We claim that every $A\in\mathcal{A}$ is contained in $Y\cup R$. It's trivial for $|A\cap Y|=t$. If $|A\cap Y|>t$, choose
$Q\in\binom{Y}{t-1}$ containing $Y\setminus A$. For
$y\in Y\setminus Q$, comparison of $A$ with
$(Q\cup\{y\})\dot\cup R$ gives
\[
  |(A\setminus Y)\cap R|
  \ge 2t-|A\cap Y|=|A\setminus Y|,
\]
so $A\setminus Y\subseteq R$. Therefore $X=Y\dot\cup R$ and $|X|=3t$.
On the other hand, for any $B\in\binom{X}{2t}$, for $t$-set $X\setminus B$, the member of $\mathcal{A}$ disjoint from
$X\setminus B$ must be $B$. Hence $\mathcal{A}=\binom{X}{2t}$.

Now suppose that every two distinct members of $\mathcal{A}$ intersect in
exactly $t$ points. Choose $S,T\in\binom Yt$ with
$|S\cap T|=t-1$. Consider the intersection of $S\dot\cup R_S$ and $T\dot\cup R_T$, $S\dot\cup R_S$ and $(Y\setminus T)\dot\cup R_T$, we get
\[
  |R_S\cap R_T|=t-|S\cap T|=1,
  \qquad
  |R_S\cap R_T|=t-|S\cap(Y\setminus T)|=t-1.
\]
Thus $1=t-1$, and therefore $t=2$.

Write $Y=\{1,2,3,4\}$. By \eqref{eq:complement}, there are three
outside $2$-sets
\[
  R_{12}=R_{34}=:R_1,\qquad
  R_{13}=R_{24}=:R_2,\qquad
  R_{14}=R_{23}=:R_3.
\]
Every $A\in\mathcal{A}\setminus\{Y\}$ meets $Y$ in exactly two points and,
by uniqueness, is one of the six sets in \eqref{eq:canonical}.
The intersection of elements in $\mathcal{A}$ gives $|R_i\cap R_j|=1$ for $i\ne j$. Moreover, $R_1\cap R_2\cap R_3=\varnothing$: otherwise, if
$u$ belonged to this intersection, then for any $y\in Y$ the
$2$-set $\{u,y\}$ would intersect with every member of
$\mathcal{A}$, contradiction. Hence $R_1,R_2,R_3$ form a
triangle. After relabelling,
$R_1=\{a,b\}$, $R_2=\{b,c\}$ and $R_3=\{c,a\}$. Consequently
$X=\{1,2,3,4,a,b,c\}$ and 
\[
  \mathcal{A}=\{1234,12ab,34ab,13bc,24bc,14ca,23ca\}.
\]
It's the family of complements of the lines of the Fano plane $\mathcal{L}=\{13a, 23b, 12c, 14b, 34c, 24a, abc\}$. This proves the second assertion.
\end{proof}

If $\mathcal F$ is a maximal extremal family, by Lemma~\ref{lem:F=sigma-A}, $\mathcal F=\Sigma_n[\mathcal A]$ where $\mathcal A$ is a $2t$-uniform $t$-intersecting family of partial permutations satisfying $\tau(\mathcal{A})>t$.  If $t=1$, then $\mathcal A$ is a $2$-uniform intersecting graph with no common vertex; hence it is a triangle, i.e. $\mathcal A=\binom X2$ for a $3$-point partial permutation $X$.

Assume $t\ge2$.  By Proposition~\ref{prop:union-partial}, the set $X=\bigcup_{A\in\mathcal A}A$ is a partial permutation. Hence every $t$-subset of $X$ is a $t$-partial permutation. Proposition~\ref{prop:A-structure} now gives exactly the two configurations stated in the theorem, from where we know the structure of maximal extremal family.

\begin{proof}[Proof of \Cref{thm:all-extremal}]
Let $\mathcal F\subseteq\Sigma_n$ be $t$-intersecting and $n\ge 130(t+1)(8t)^t$,  suppose that $\gamma_t(\mathcal F)=D_t(n)$.  Extend $\mathcal F$ to a maximal $t$-intersecting family $\mathcal F_{\max}$.  Since $t$-diversity is monotone under inclusion,
\[
        D_t(n)=\gamma_t(\mathcal F)
        \le \gamma_t(\mathcal F_{\max}).
\]
By \Cref{thm:main-bound}, $\gamma_t(\mathcal F_{\max})\le D_t(n)$, so equality holds for $\mathcal F_{\max}$.  Applying Lemma~\ref{lem:F=sigma-A}, we have
\[
        \mathcal F\subseteq\mathcal F_{\max}=\Sigma_n[\mathcal A],
\]
where $\mathcal A$ is one of the two finite configurations described in Proposition~\ref{prop:A-structure}.

\item Case 1: $\mathcal A=\binom X{2t}$ for a partial permutation $X$ of size $3t$.

Then
\[
        \Sigma_n[\mathcal A]=\mathcal{T}(X,t)=\{F\in\Sigma_n:|F\cap X|\ge2t\},
\]
so the desired upper inclusion holds.

Then we prove the lower inclusion.  For each $A\in\binom X{2t}$, put $T_A=X\setminus A$.  If $F\in\mathcal F(\overline{T_A})$, then $F\in\Sigma_n[\mathcal A]$ and $F$ is disjoint from $T_A$.  Hence the only possible $2t$-subset of $X$ contained in $F$ is $A$, and therefore
\[
        \mathcal F(\overline{T_A})
        =\mathcal F[A,\overline{T_A}]
        \subseteq \Sigma_n[A,\overline{T_A}].
\]
By \Cref{lem:counting}, since $A\cup T_A=X$ is a partial permutation,
\[
        |\Sigma_n[A,\overline{T_A}]|=D_t(n).
\]
On the other hand, $|\mathcal F(\overline{T_A})|\ge\gamma_t(\mathcal F)=D_t(n)$.  Thus
\[
        \mathcal F[A,\overline{T_A}]=\Sigma_n[A,\overline{T_A}]
\]
for every $A\in\binom X{2t}$.  Taking the union over all such $A$ gives
\[
        \mathcal{T}^*(X,t)=\{F\in\Sigma_n:|F\cap X|=2t\}
        \subseteq\mathcal F.
\]
By monotonicity, any family lying between the two families also attains the extremal diversity.
This proves the first structure in \Cref{thm:all-extremal}.

\item Case 2: $t=2$ and $\mathcal A=\{V\setminus L:L\in\mathcal L\}$ for a Fano plane $(V,\mathcal L)$.

The upper inclusion is
\[
        \mathcal F\subseteq\Sigma_n[\mathcal A]
        =\{F\in\Sigma_n:F\supseteq V\setminus L\text{ for some }L\in\mathcal L\}.
\]
We prove the lower inclusion.  Fix a line $L=\{a,b,c\}$.  For the pair $P=\{a,b\}$, the unique member of $\mathcal A$ disjoint from $P$ is $V\setminus L$, because $a$ and $b$ lie on the unique Fano line $L$.  Hence, using $\mathcal F\subseteq\Sigma_n[\mathcal A]$,
\[
        \mathcal F(\overline P)
        =\mathcal F[V\setminus L,\overline P]
        \subseteq
        \Sigma_n[V\setminus L,\overline P].
\]
The set $(V\setminus L)\cup P$ is a partial permutation of size $6$, so Lemma~\ref{lem:counting} gives
\[
        |\Sigma_n[V\setminus L,\overline P]|=D_2(n).
\]
Since $|\mathcal F(\overline P)|\ge\gamma_2(\mathcal F)=D_2(n)$, equality holds and
\[
        \mathcal F[V\setminus L,\overline P]
        =\Sigma_n[V\setminus L,\overline P].
\]
The same argument applies to the pairs $\{b,c\}$ and $\{c,a\}$.  Taking the union over the three pairs on $L$, we obtain
\[
        \{F\in\Sigma_n:F\supseteq V\setminus L\text{ and }|F\cap L|\le1\}
        \subseteq\mathcal F.
\]
Finally take the union over all Fano lines $L\in\mathcal L$.  This gives
\[
        \{F\in\Sigma_n:
        F\supseteq V\setminus L\text{ and }|F\cap L|\le1
        \text{ for some }L\in\mathcal L\}
        \subseteq\mathcal F,
\]
as required.
\end{proof}

\section{The another version of $t$-diversity: Proof of \Cref{thm:main2}}

\subsection{The upper bound}

We first calculate the star $t$-diversity of the proposed construction. 
Let
\[
\mathcal{C}(X,t)=\{F\in\Sigma_n:|F\cap X|= t+1\}
\qquad
\mathcal{A}(X,t)=\{F\in\Sigma_n:|F\cap X|\geq t+1\}
\]
where $X$ is a partial permutation of size $t+2$.

\begin{lemma}\label{lem:model}
If $2(n-t-2)\ge t+2$, then
\[
  \gamma_t^{\star}(\mathcal{C}(X,t))
  =\gamma_t^{\star}(\mathcal{A}(X,t))
  =t(n-t-2)(n-t-2)!.
\]
\end{lemma}
\begin{proof}
Take $T\in\binom Xt$.  A member of $\mathcal{A}(X,t)$ (or $\mathcal{C}(X,t)$) fails to contain $T$ precisely when it lies in $\mathcal C_x$ for some $x\in T$, where $\mathcal C_x=\{ F \in \Sigma_n:F\cap X=X\setminus \{x\} \}$.  Hence
\[
  |\mathcal{A}(X,t)\setminus\mathcal{A}(X,t)[T]|
  =|\mathcal{C}(X,t)\setminus\mathcal{C}(X,t)[T]|
  = \sum_{x\in T}|\mathcal C_x|=t(n-t-2)(n-t-2)!.
\]
By the definition of star $t$-diversity, this means 
\[
  \gamma_t^{\star}(\mathcal{A}(X,t))
  \leq t(n-t-2)(n-t-2)!
  \qquad
  \gamma_t^{\star}(\mathcal{C}(X,t))
  \leq t(n-t-2)(n-t-2)!.
\]
It remains to show that no partial permutation $T$ outside $\binom Xt$ gives a smaller value for $\mathcal{C}(X,t)$.  If $T\not\subseteq X$, then for every $x\in X$ either $T\cup(X\setminus\{x\})$ is not a partial permutation or it has at least $t+2$ points.  Consequently
\[
  |\mathcal C_x[T]|\leq(n-t-2)!.
\]
Hence,
\begin{align*}
  |\mathcal{C}(X,t)\setminus\mathcal{C}(X,t)[T]|
  &\geq(t+2)\bigl(|\mathcal C_x|-(n-t-2)!\bigr)\\
  &\geq t((n-t-1)!-(n-t-2)!)\\
  &=t(n-t-2)(n-t-2)!.
\end{align*}
This proves the assertion for $\mathcal{C}(X,t)$.  Since $\mathcal{C}(X,t)\subseteq\mathcal{A}(X,t)$ and star $t$-diversity is monotone under inclusion, the assertion for $\mathcal{A}(X,t)$ follows from the upper value already witnessed by every $T\in\binom Xt$.
\end{proof}

By the same argument of \cref{lem:large-fibre}, we get the following lemma.
\begin{lemma}\label{lem:large-fiber}
Let $t\ge 1$ and $n\ge t+7$, $B$ be a partial permutation of size $t+1$, and let $\mathcal{F}\subseteq\Sigma_n$ be $t$-intersecting.  If
\[
  |\mathcal{F}[B]|\geq\left(1-\frac1e+0.001\right)(n-t-1)!,
\]
then every $F\in\mathcal{F}$ satisfies $|F\cap B|\geq t$.

\end{lemma}

Next we collect below the numerical consequences of the lower bound on $n$
that will be used soon. 
\begin{lemma}\label{Counting2}
Let $n\ge 6(t+4)^4, t\ge 1$ and $x=4.1,\tau=1.5,r_0=\frac{n-q}{e},q=\left\lceil(t+2)\log_{\tau}n\right\rceil$, put $\rho_t:=2(t+2)(t+1)^{3}+\frac{2}{3}$, then
\begin{enumerate}[(i)]
    \item $r_0\geq x\tau q$ and $r_0>2^5\frac{x\tau}{x-1}\log_2(2n)$
 
    \item $n>24q$ and  \[
  \begin{aligned}
    &t(n-t-2)>\rho_t,\\
    &t(n-t-2)>(t-1)(n-t-1)+\rho_t,\\
    &t(n-t-2)>(t-1)(n-t-2)+\left(1-\frac1e+0.001\right)(n-t-1)+\rho_t.
  \end{aligned}
  \]
\end{enumerate}
\end{lemma}
\begin{proof}

We first prove $n>24q$, which will be used below.  

Since $f(u):=0.50\sqrt u-\log_{1.5}u$ is increasing and positive for $u\ge3750$, we have
\[
  \log_{\tau}n<0.50\sqrt n.
\]
Moreover, since
$2.44(t+4)^2\ge20(t+2)$, we have
\[
  \sqrt n\ge\sqrt6\,(t+4)^2\ge20(t+2).
\]
Thus
\[
\begin{aligned}
  24q
  &\le24\bigl((t+2)\log_{\tau}n+1\bigr)\\
  &<12(t+2)\sqrt n+24
   \le0.60n+24<n.
\end{aligned}
\]
Consequently,
\[
  r_0=\frac{n-q}{e}>\frac{23q}{e}>6.15q=x\tau q
  \qquad and \qquad
   r_0>\frac{23n}{24e}>0.35n.
\]
Since $ g(u):=u-182\log_2(2u)$ is increasing and positive for $u\ge3750$, we have
\[
  182\log_2(2n)<n.
\]
Therefore,
\[
  2^5\frac{x\tau}{x-1}\log_2(2n)
  <63.49\log_2(2n)<0.35n<r_0.
\]
This proves (i).

Furthermore,
\[
\begin{aligned}
  n-t-2-6(t+2)(t+1)^3
  &\ge6(t+4)^4-t-2-6(t+2)(t+1)^3>0.
\end{aligned}                                     
\]
Thus,
\[
  t(n-t-2)>n-t-2>6(t+2)(t+1)^3>\rho_t,
\]
and
\[
\begin{aligned}
 &t(n-t-2)-\bigl((t-1)(n-t-1)+\rho_t\bigr)\\
 &\qquad=n-t-2-\rho_t-t+1\\
 &\qquad>4(t+2)(t+1)^3-t+0.33>0.
\end{aligned}
\]
Finally, it's easy to check
\[
\begin{aligned}
  \left(\frac1e-0.001\right)(n-t-2)
  >2.16(t+2)(t+1)^3>\rho_t+1-\frac1e+0.001.
\end{aligned}
\]
which is equivalent to,
\[
  (t-1)(n-t-2)
  +\left(1-\frac1e+0.001\right)(n-t-1)+\rho_t
  <t(n-t-2).
\]
This proves (ii).  
\end{proof}

\begin{proof}[Proof of the upper bound of \Cref{thm:main2}]
First, we apply \cref{lem:spread-approximation}  to $\mathcal F$ with $x=4.1,\tau=1.5,r_0=\frac{n-q}{e},q=\left\lceil(t+2)\log_{\tau}n\right\rceil$, where all the conditions of the lemma are satisfied by \cref{Counting2} $(i)$, obtaining that there exists a $t$-intersecting family $\mathcal B$ of partial permutations, each of size at most $q$, and $\mathcal F'\subseteq\mathcal F$ such that
\[
        \mathcal F=\mathcal F[\mathcal B]\cup\mathcal F'
 \qquad and\qquad
        |\mathcal F'|\le\tau^{-1}n^{-(t+2)}n! \le\frac23(n-t-2)!.
\]

Then we use the peeling operation from Section~3 again, but this time we peel until $t+2$.
Set $\mathcal P_q=\mathcal R(\mathcal B,q)$ and
\[
\mathcal A_q=\{A\in\mathcal P_q:|A|=q\}.
\]
For $i=q-1,q-2,\ldots,t+2$, define
\[
\mathcal P_i=
\mathcal R(\mathcal P_{i+1}\setminus\mathcal A_{i+1},i),
\qquad
\mathcal A_i=\{A\in\mathcal P_i:|A|=i\}.
\]
Finally, put
\[
\mathcal A=\mathcal P_{t+2}\setminus\mathcal A_{t+2}.
\]
Then $\mathcal A$ is a $t$-intersecting family of partial permutations, each of size at most $t+1$. By construction, every $B\in\mathcal B$ contains some member of $\mathcal A\cup\mathcal A_{t+2}\cup\cdots\cup\mathcal A_q$.  Hence
\[
        \mathcal F[\mathcal B]
        \subseteq
        \mathcal F[\mathcal A]\cup
        \bigcup_{i=t+2}^{q}\mathcal F[\mathcal A_i].
\]
Let
\[
        \mathcal F''=\bigcup_{i=t+2}^{q}\mathcal F[\mathcal A_i]\cup\mathcal F'.
\]

Because of \cref{Counting2} $(ii)$, through similar simple calculations, we can get the scaling in \eqref{eq:star} again.
Therefore, for $n\ge6(t+4)^4$
\[
\begin{aligned}
        |\mathcal F''|
        &\le \sum_{i=t+2}^{q}|\mathcal A_i|(n-i)!+|\mathcal F'| \\
        &\le 4\binom{t+2}{t}(t+1)^{2}(n-t-2)!+\frac{2}{3}(n-t-2)! \\
        &=(2(t+2)(t+1)^{3}+\frac{2}{3})(n-t-2)!=\rho_t(n-t-2)!\\
        &<t(n-t-2)(n-t-2)!.
\end{aligned}
\]

If there exists a $t$-partial permutation $T$ contained by every member of $\mathcal A$, then $\{F\in \mathcal F[\mathcal A]\colon T \not\subset F\}=\emptyset$, so
\[
        \gamma_t^{\star}(\mathcal F)
        \le |\mathcal{F}\backslash\mathcal{F}[T]|
        \le |\mathcal F''|
        =\rho_t(n-t-2)!<t(n-t-2)(n-t-2)!,
\]
and \Cref{thm:main2} follows.  Hence we may assume from now on that
\begin{equation*}
        \text{for every $t$-partial permutation $T$, there exists $A\in\mathcal A$ such that $T\not\subset A$.}
\end{equation*}
Let us recall the $t$-transversal number $\tau_t(\mathcal F):=\min\bigl\{|T|:\ T\text{ is a $t$-transversal of }\mathcal F\bigr\}$, so the condition above is equivalent to $\tau_t(\mathcal A)>t$.

\begin{claim}\label{claim:t+1-uniform}
The family $\mathcal A$ is $(t+1)$-uniform and there is a set $X$ of size $t+2$ such that
  \[
    \mathcal{A}\subseteq\binom X{t+1}.
  \]
  Moreover, if $\mathcal{A}\subsetneq \binom X{t+1}$, there is a partial permutation $T$ of size $t$ contained in all but at most $t-1$ members of $\mathcal{A}$.  If $\mathcal{A}=\binom X{t+1}$, then $X$ itself is a partial permutation.
\end{claim}

\begin{proof}
It's easy to know that every member of $\mathcal A$ has size $t$ or $t+1$. 
If there exists a $T \in\mathcal{A}$ of size t, then the $t$-intersetion of $\mathcal{A}$ gives that every member contains T, which contradicts $\tau_t(\mathcal A)>t$. Hence every member of $\mathcal A$ has size exactly $t+1$.

Take two distinct members $A_0,A_1\in \mathcal{A}$ such that their intersection $C=A_0\cap A_1$ has size $t$, their union $X=A_0\cup A_1$ has size $t+2$. Moreover, a member $C\cup\{z\}$ with $z\notin X$ has only $t-1$ points in common with every member $X\setminus\{c\}$ that does not contain $C$.  Hence the two types cannot mix: the whole family is contained either in the $t$-star centered in $C$ or in the $\binom X{t+1}$. Since whenever $\mathcal{A}$ is contained in a $t$-star will contradict $\tau_t(\mathcal A)>t$, we get  $\mathcal{A}\subseteq\binom X{t+1}$. 

Hence we can write
\[
  \mathcal{A}=\{X\setminus\{x\}:x\in S\}
\]
for some $S\subseteq X$.  If $|S|\leq2$, the common intersection contains a $t$-set.  If $3\leq|S|\leq t+1$, choose distinct $u,v\in S$ and put $T=X\setminus\{u,v\}$.  This is a partial permutation because it equals
\[
  (X\setminus\{u\})\cap(X\setminus\{v\}).
\]
Exactly $|S|-2\leq t-1$ members of $\mathcal{A}$ fail to contain $T$.

Finally, if $S=X$, then every $(t+1)$-subset of $X$ is a partial permutation.  Any conflicting pair of points of $X$ would lie together in one of these subsets, since $|X|=t+2\geq3$, a contradiction. Thus, $X$ is a partial permutation.
\end{proof}

\begin{lemma} \label{lem:recognition}
Let $X$ be a partial permutation of size $t+2$.  A permutation $F\in\Sigma_n$ $t$-intersects with every member of $\binom X{t+1}$ if and only if $|F\cap X|\geq t+1$, that is to say $F\in \mathcal{A}(X,t)$.
\end{lemma}

\begin{proof}
Obviously, $|F\cap X|\geq t$.  If $|F\cap X|=t$, choose $x\in F\cap X$, then $X\setminus\{x\}\in\binom X{t+1}$ and $|F\cap(X\setminus\{x\})|=t-1$, contradiction. Thus, $|F\cap X|\geq t+1$ is necessary.

Conversely, if $|F\cap X|\geq t+1$, then for every $x\in X$, removing $x$ decreases the intersection with X by at most one, so F meets $X\backslash \{x\}$ in at least t points.
\end{proof}

By \Cref{claim:t+1-uniform}, $\mathcal{A}$ is contained in $\binom X{t+1}$.  If $\mathcal{A}\neq \binom X{t+1}$, then there is a partial permutation $T$ of size $t$ contained in all but at most $t-1$ elements of $\mathcal{A}$. Hence, by \cref{Counting2} $(ii)$, for $n\geq 6(t+4)^4$
\begin{align*}
   \gamma_t^\star(\mathcal{F})
  & \leq|\mathcal{F}\backslash\mathcal{F}[T]|\leq(t-1)(n-t-1)!+|\mathcal F''|\\
  &\leq(t-1)(n-t-1)!+\rho_t(n-t-2)!\\
  &<t(n-t-2)(n-t-2)!.
\end{align*}
Thus  $\mathcal{A}=\binom X{t+1}$, where $X$ is a partial permutation with $t+2$ points.  For $x\in X$, let
$A_x=X\setminus\{x\}$.

First we assume that there exists $x\in X$ such that $ |\mathcal{F}[A_x]|<\left(1-\frac1e+0.001\right)(n-t-1)!.$ Choose a partial permytation $T\in\binom Xt$ with $x\in T$.  Exactly $t$ elements
 of $\mathcal{A}=\binom X{t+1}$ fail to contain $T$, we write $A_y$, with $y\in T$.  For every $y\neq x$, we have
\[
  |\mathcal{F}[A_y]\backslash\mathcal{F}[T]|\leq |\Sigma_n[A_y]\setminus\Sigma_n[T]|=(n-t-1)!-(n-t-2)!, 
\]
while the $\mathcal{F}[A_x]$ contributes less than $\left(1-\frac1e+0.001\right)(n-t-1)!$.  

Hence, by \cref{Counting2} $(ii)$, for $n\geq 6(t+4)^4$
\begin{align*}
  \gamma_t^\star(\mathcal{F})
  &\leq(t-1)((n-t-1)!-(n-t-2)!)+\left(1-\frac1e+0.001\right)(n-t-1)!+\rho_t(n-t-2)!\\
  &<t(n-t-2)(n-t-2)!.
\end{align*}

It remains to treat the case that for every $x\in X$, $|\mathcal{F}[A_x]|\geq \left(1-\frac1e+0.001\right)(n-t-1)!.$ By \Cref{lem:large-fiber}, every $F\in\F$ meets every $A_x$ in at least
$t$ points.  Then by \Cref{lem:recognition},
\[
  \mathcal{F}\subseteq\mathcal{A}(X,t).
\]
Using monotonicity of star $t$-diversity under inclusion and Lemma
\ref{lem:model}, we conclude that
\[
  \gamma_t^\star(\mathcal{F})
  \leq\gamma_t^\star(\mathcal{A}(X,t))
  =t(n-t-2)(n-t-2)!.
\]
This proves the upper bound of \Cref{thm:main2}.
\end{proof}

\subsection{The structure of the equality cases}
Now we consider $\mathcal{F}\subset \Sigma_n$ satisfying that $\gamma_t^\star(\mathcal{F})=t(n-t-2)(n-t-2)!$.

Three cases in the proof of the upper bound give strict inequalities: $\mathcal{A}$ being contained in a $t$-star, $\mathcal{A}\subsetneq \binom X{t+1}$ and $\mathcal{A}=\binom X{t+1}$ with some $x\in X$ such that $ |\mathcal{F}[A_x]|<\left(1-\frac1e+0.001\right)(n-t-1)!.$  Hence only the case $\mathcal{A}=\binom X{t+1}$ with all $x\in X$ such that $ |\mathcal{F}[A_x]|\geq \left(1-\frac1e+0.001\right)(n-t-1)!$ is possible.  

Thus, for some partial permutation $X$ of size $t+2$, we have 
\begin{equation*}
  \mathcal{F}\subseteq\mathcal{A}(X,t).
\end{equation*}
Fix $T\in\binom Xt$.  By \cref{lem:model},
\[
  \mathcal{A}(X,t)\setminus\mathcal{A}(X,t)[T]
  =\mathop{\dot\bigcup}_{x\in T}\mathcal C_x,
\]
and this set has exactly $t|\mathcal C_x|=t(n-t-2)(n-t-2)!$ members.  On the other hand, $\mathcal{F}\subseteq\mathcal{A}(X,t)$ gives 
\[
  \mathcal{F}\setminus\mathcal{F}[T]
  \subseteq
  \mathop{\dot\bigcup}_{x\in T}\mathcal C_x.
\]
However, by the definition of star $t$-diversity, every partial permutation $T$ of size $t$ satisfies
\[
  |\mathcal{F}\setminus\mathcal{F}[T]|
  \geq\gamma_t^\star(\mathcal{F})=t(n-t-2)(n-t-2)!.
\]
Then the inclusion must be an equality, and hence
\[
  \mathcal C_x\subseteq\mathcal{F}\qquad \forall x\in T.
\]
Letting $T$ range over $\binom Xt$, we obtain
\[
  \mathcal{C}(X,t)=\mathop{\dot\bigcup}_{x\in X}\mathcal C_x
  \subseteq\mathcal{F}\subseteq\mathcal{A}(X,t).
\]

Conversely, suppose that a $t$-intersecting family $\mathcal{F}\subset \Sigma_n$ satisfies this interval of inclusions. \Cref{lem:model} and monotonicity of star $t$-diversity imply
\[
  t(n-t-2)(n-t-2)!
  =\gamma_t^\star(\mathcal{C}(X,t))
  \leq\gamma_t^\star(\mathcal{F})
  \leq\gamma_t^\star(\mathcal{A}(X,t))
  =t(n-t-2)(n-t-2)!.
\]
Thus equality holds. 

In particular, if $\mathcal{F}$ is also maximal, then $\mathcal{F}=\{F\in \Sigma_n:|F\cap X|\geq t+1\}=\Sigma_n[\mathcal A]$ where $\mathcal{A}=\binom{X}{t+1}$ and $X$ is still a partial permutation $X$ of size $t+2$.

\bibliographystyle{abbrv}
\bibliography{reference}

@article{KUPAVSKII2024109653,
  author   = {Kupavskii, Andrey and Zakharov, Dmitrii},
  title    = {Spread approximations for forbidden intersections problems},
  journal  = {Advances in Mathematics},
  volume   = {445},
  pages    = {109653},
  year     = {2024},
  issn     = {0001-8708},
  doi      = {10.1016/j.aim.2024.109653},
  url      = {https://www.sciencedirect.com/science/article/pii/S0001870824001683}
}

@article{FUREDI1980282,
  author   = {F{\"u}redi, Zolt{\'a}n},
  title    = {On maximal intersecting families of finite sets},
  journal  = {Journal of Combinatorial Theory, Series A},
  volume   = {28},
  number   = {3},
  pages    = {282--289},
  year     = {1980},
  issn     = {0097-3165},
  doi      = {10.1016/0097-3165(80)90071-0},
  url      = {https://www.sciencedirect.com/science/article/pii/0097316580900710}
}

@article{EKR,
  author  = {Erd{\H{o}}s, P. and Ko, C. and Rado, R.},
  title   = {Intersection theorems for systems of finite sets},
  journal = {Quart. J. Math. Oxford Ser. (2)},
  volume  = {12},
  year    = {1961},
  pages   = {313--320},
  doi     = {10.1093/qmath/12.1.313}
}

@article{DezaFrankl77,
  author  = {Deza, M. and Frankl, P.},
  title   = {On the maximum number of permutations with given maximal or minimal distance},
  journal = {J. Combin. Theory Ser. A},
  volume  = {22},
  number  = {3},
  year    = {1977},
  pages   = {352--360},
  doi     = {10.1016/0097-3165(77)90009-7}
}

@article{CameronKu03,
  author  = {Cameron, P. J. and Ku, C. Y.},
  title   = {Intersecting families of permutations},
  journal = {European J. Combin.},
  volume  = {24},
  number  = {7},
  year    = {2003},
  pages   = {881--890},
  doi     = {10.1016/S0195-6698(03)00078-7}
}

@article{LaroseMalvenuto04,
  author  = {Larose, B. and Malvenuto, C.},
  title   = {Stable sets of maximal size in {Kneser}-type graphs},
  journal = {European J. Combin.},
  volume  = {25},
  year    = {2004},
  pages   = {657--673},
  doi     = {10.1016/j.ejc.2003.10.003}
}

@article{EllisFriedgutPilpel11,
  author  = {Ellis, D. and Friedgut, E. and Pilpel, H.},
  title   = {Intersecting families of permutations},
  journal = {J. Amer. Math. Soc.},
  volume  = {24},
  number  = {3},
  year    = {2011},
  pages   = {649--682},
  doi     = {10.1090/S0894-0347-2011-00690-5}
}

@article{EllisKellerLifshitz24,
  author  = {Ellis, D. and Keller, N. and Lifshitz, N.},
  title   = {Stability for the complete intersection theorem, and the forbidden intersection problem of {Erd{\H{o}}s} and {S{\'o}s}},
  journal = {J. Eur. Math. Soc.},
  volume  = {26},
  number  = {5},
  year    = {2024},
  pages   = {1611--1654},
  doi     = {10.4171/JEMS/1346}
}

@article{Kupavskii24b,
  author        = {Kupavskii, A.},
  title         = {An almost complete {$t$}-intersection theorem for permutations},
  journal       = {arXiv preprint},
  year          = {2024},
  eprint        = {2405.07843},
  archivePrefix = {arXiv}
}

@article{HiltonMilner67,
  author  = {Hilton, A. J. W. and Milner, E. C.},
  title   = {Some intersection theorems for systems of finite sets},
  journal = {Quart. J. Math. Oxford Ser. (2)},
  volume  = {18},
  year    = {1967},
  pages   = {369--384},
  doi     = {10.1093/qmath/18.1.369}
}

@incollection{Frankl77,
  author    = {Frankl, P.},
  title     = {The {Erd{\H{o}}s--Ko--Rado} theorem is true for {$n=ckt$}},
  booktitle = {Combinatorics, Vol. I, Proceedings of the Fifth Hungarian Colloquium, Keszthely, 1976},
  series    = {Colloq. Math. Soc. J{\'a}nos Bolyai},
  volume    = {18},
  pages     = {365--375},
  publisher = {North-Holland},
  address   = {Amsterdam},
  year      = {1978}
}

@article{Frankl87,
  author  = {Frankl, P.},
  title   = {{Erd{\H{o}}s--Ko--Rado} theorem with conditions on the maximal degree},
  journal = {J. Combin. Theory Ser. A},
  volume  = {46},
  number  = {2},
  year    = {1987},
  pages   = {252--263},
  doi     = {10.1016/0097-3165(87)90005-7}
}

@article{Kupavskii18,
  author  = {Kupavskii, A.},
  title   = {Diversity of uniform intersecting families},
  journal = {European J. Combin.},
  volume  = {74},
  year    = {2018},
  pages   = {39--47},
  doi     = {10.1016/j.ejc.2018.06.004}
}

@article{FranklWang24,
  author  = {Frankl, P. and Wang, J.},
  title   = {Improved bounds on the maximum diversity of intersecting families},
  journal = {European J. Combin.},
  volume  = {118},
  year    = {2024},
  pages   = {103885},
  doi     = {10.1016/j.ejc.2024.103885}
}

@article{WangXiao25,
  author  = {Wang, J. and Xiao, J.},
  title   = {A note on the maximum diversity of intersecting families in the symmetric group},
  journal = {European J. Combin.},
  volume  = {134},
  year    = {2026},
  pages   = {104331},
  doi     = {10.1016/j.ejc.2025.104331}
}

@article{Frankl22,
  author  = {Frankl, P.},
  title   = {Pseudo sunflowers},
  journal = {European J. Combin.},
  volume  = {104},
  year    = {2022},
  pages   = {103553},
  doi     = {10.1016/j.ejc.2022.103553}
}

@article{AlweissLovettWuZhang21,
  author  = {Alweiss, R. and Lovett, S. and Wu, K. and Zhang, J.},
  title   = {Improved bounds for the sunflower lemma},
  journal = {Ann. of Math. (2)},
  volume  = {194},
  number  = {3},
  year    = {2021},
  pages   = {795--815},
  doi     = {10.4007/annals.2021.194.3.5}
}

@article{Wilson84,
  author  = {Wilson, R. M.},
  title   = {The exact bound in the {Erd{\H{o}}s--Ko--Rado} theorem},
  journal = {Combinatorica},
  volume  = {4},
  number  = {2--3},
  year    = {1984},
  pages   = {247--257},
  doi     = {10.1007/BF02579226}
}

@article{AhlswedeKhachatrian97,
  author  = {Ahlswede, R. and Khachatrian, L. H.},
  title   = {The complete intersection theorem for systems of finite sets},
  journal = {European J. Combin.},
  volume  = {18},
  number  = {2},
  year    = {1997},
  pages   = {125--136},
  doi     = {10.1006/eujc.1995.0092}
}

@article{Frankl78,
  author  = {Frankl, P.},
  title   = {On intersecting families of finite sets},
  journal = {J. Combin. Theory Ser. A},
  volume  = {24},
  number  = {2},
  year    = {1978},
  pages   = {146--161},
  doi     = {10.1016/0097-3165(78)90003-1}
}

@incollection{FranklShifting87,
  author    = {Frankl, P.},
  title     = {The shifting technique in extremal set theory},
  booktitle = {Surveys in Combinatorics 1987},
  series    = {London Math. Soc. Lecture Note Ser.},
  volume    = {123},
  pages     = {81--110},
  publisher = {Cambridge Univ. Press},
  address   = {Cambridge},
  year      = {1987}
}

@article{FranklFuredi86,
  author  = {Frankl, P. and F{\"u}redi, Z.},
  title   = {Non-trivial intersecting families},
  journal = {J. Combin. Theory Ser. A},
  volume  = {41},
  number  = {1},
  year    = {1986},
  pages   = {150--153},
  doi     = {10.1016/0097-3165(86)90121-4}
}

@article{GodsilMeagher09,
  author  = {Godsil, C. and Meagher, K.},
  title   = {A new proof of the {Erd{\H{o}}s--Ko--Rado} theorem for intersecting families of permutations},
  journal = {European J. Combin.},
  volume  = {30},
  number  = {2},
  year    = {2009},
  pages   = {404--414},
  doi     = {10.1016/j.ejc.2008.05.006}
}

@article{KellerLifshitzMinzerSheinfeld24,
  author  = {Keller, N. and Lifshitz, N. and Minzer, D. and Sheinfeld, O.},
  title   = {On {$t$}-intersecting families of permutations},
  journal = {Advances in Mathematics},
  volume  = {445},
  year    = {2024},
  pages   = {109650},
  doi     = {10.1016/j.aim.2024.109650}
}

@article{LemonsPalmer08,
  author  = {Lemons, N. and Palmer, C.},
  title   = {The unbalance of set systems},
  journal = {Graphs Combin.},
  volume  = {24},
  number  = {4},
  year    = {2008},
  pages   = {361--365},
  doi     = {10.1007/s00373-008-0793-9}
}

@article{Frankl20,
  author  = {Frankl, P.},
  title   = {Maximum degree and diversity in intersecting hypergraphs},
  journal = {J. Combin. Theory Ser. B},
  volume  = {144},
  year    = {2020},
  pages   = {81--94}
}

@article{FranklKupavskii21,
  author  = {Frankl, P. and Kupavskii, A.},
  title   = {Diversity},
  journal = {J. Combin. Theory Ser. A},
  volume  = {182},
  year    = {2021},
  pages   = {105468},
  doi     = {10.1016/j.jcta.2021.105468}
}

@article{FranklWang24MaxDegree,
  author  = {Frankl, P. and Wang, J.},
  title   = {Improved bounds concerning the maximum degree of intersecting hypergraphs},
  journal = {Electron. J. Combin.},
  volume  = {31},
  number  = {2},
  year    = {2024},
  pages   = {Paper No. 2.33},
  doi     = {10.37236/11616}
}

@article{FranklWang25Cdiversity,
  author  = {Frankl, P. and Wang, J.},
  title   = {On the {$C$}-diversity of intersecting {$k$}-graphs},
  journal = {European J. Combin.},
  volume  = {130},
  year    = {2025},
  pages   = {104199},
  doi     = {10.1016/j.ejc.2025.104199}
}

@article{KuWong20,
  author  = {Ku, C. Y. and Wong, K. B.},
  title   = {On diversity of certain {$t$}-intersecting families},
  journal = {Bull. Korean Math. Soc.},
  volume  = {57},
  number  = {4},
  year    = {2020},
  pages   = {815--829},
  doi     = {10.4134/BKMS.b190301}
}

@article{MagnanPalmerWood24,
  author  = {Magnan, V. and Palmer, C. and Wood, R.},
  title   = {A generalization of diversity for intersecting families},
  journal = {European J. Combin.},
  volume  = {122},
  year    = {2024},
  pages   = {104041},
  doi     = {10.1016/j.ejc.2024.104041}
}

@misc{Stoeckl22,
  author       = {Stoeckl, M.},
  title        = {Lecture notes on recent improvements for the sunflower lemma},
  year         = {2022},
  howpublished = {\url{https://mstoeckl.com/notes/research/sunflower_notes.html}}
}

\end{document}